\documentclass[a4paper]{amsart}
\usepackage{amssymb}
\usepackage{amsmath}
\usepackage{amscd}
\usepackage{mathtools}

\newtheorem{theorem}{Theorem}
\newtheorem*{theorem*}{Theorem}
\newtheorem{proposition}[theorem]{Proposition}
\newtheorem*{proposition*}{Proposition}
\newtheorem{corollary}[theorem]{Corollary}
\newtheorem{lemma}[theorem]{Lemma}
\theoremstyle{definition}
\newtheorem{definition}[theorem]{Definition}

\theoremstyle{remark}
\newtheorem{remark}[theorem]{Remark}

\numberwithin{equation}{section} \numberwithin{figure}{section}

\DeclareMathOperator{\Pic}{Pic} 
\DeclareMathOperator{\Gal}{Gal} 
 \DeclareMathOperator{\Spec}{Spec}
 
\DeclareMathOperator{\Hom}{Hom}

\DeclareMathOperator{\Br}{Br} 

\DeclareMathOperator{\inv}{inv}

\newcommand{\kbar}{{\bar{k}}}
\newcommand{\Xbar}{\bar{X}}

\newcommand{\Adele}{\mathbf{A}}

\newcommand\FF{\mathbb{F}}
\newcommand\PP{\mathbb{P}}

\newcommand\ZZ{\mathbb{Z}}

\newcommand\QQ{\mathbb{Q}}

\newcommand\GG{\mathbb{G}}

\newcommand\Gm{\GG_\mathrm{m}}
\newcommand\OO{\mathcal{O}}

\newcommand{\cA}{\mathcal{A}}
\newcommand{\fo}{\mathfrak{o}}
\newcommand{\fp}{\mathfrak{p}}
\newcommand{\mmu}{\boldsymbol{\mu}}
\renewcommand{\H}{\mathrm{H}}
\newcommand{\ff}[1]{{\kappa(#1)}}

\newcommand{\Zn}{\ZZ/n}
\newcommand{\isom}{\cong}
\newcommand{\QZ}{\QQ/\ZZ}
\newcommand{\sX}{\mathcal{X}}

\newcommand{\sD}{\mathcal{D}}
\newcommand{\sV}{\mathcal{V}}

\newcommand{\sB}{\mathcal{B}}

\newcommand{\sP}{\mathcal{P}}
\newcommand{\p}{\mathfrak{p}}
\newcommand{\R}{\mathrm{R}}
\newcommand{\et}{\mathrm{\acute{e}t}}
\DeclareMathOperator{\spec}{sp}

\newcommand{\sm}{\mathrm{sm}}

\newcommand{\dual}[1]{{#1}^\vee}
\newcommand{\sHom}{{\mathcal{H}\!\mathit{om}}}
\DeclareMathOperator{\catsh}{Sh}
\DeclareMathOperator{\catpsh}{PSh}
\newcommand{\cC}{\mathcal{C}}

\author{Martin Bright}
\address{Mathematisch Instituut, Niels Bohrweg 1, 2333 CA Leiden, Netherlands}
\email{m.j.bright@math.leidenuniv.nl}

\title{Obstructions to the Hasse principle in families}

\begin{document}

\begin{abstract}
For a family of varieties over a number field, we give conditions under which 100\% of members have no Brauer--Manin obstruction to the Hasse principle.
\end{abstract}

\maketitle

\section{Introduction}

In this article, we study the Hasse principle and the Brauer--Manin obstruction to it, in the context of a family of varieties.  We begin by briefly recalling these concepts; for more details, see~\cite{Skorobogatov:TRP}.  Let $Y$ be a smooth, projective, geometrically irreducible variety over a number field $k$.  Let $\Adele_k$ denote the ring of ad\`eles of $k$.  A necessary condition for $Y$ to have a $k$-rational point is that $Y$ have points over every completion of $k$; equivalently, that $Y(\Adele_k)$ be non-empty.  A class $\cC$ of varieties is said to satisfy the \emph{Hasse principle} if, for varieties in $\cC$, this necessary condition is also sufficient: in other words, for all $Y \in \cC$, the implication 
\[
Y(\Adele_k) \neq \emptyset \quad \Rightarrow \quad Y(k) \neq \emptyset
\]
holds.  The class $\cC$ satisfies \emph{weak approximation} if, for $Y \in \cC$, the diagonal image of $Y(k)$ is dense in $Y(\Adele_k)$.

Not all varieties do satisfy the Hasse principle, and Manin~\cite{Manin:GBG} used the Brauer group to explain many such counterexamples to the Hasse principle.  Define the Brauer group of the variety $Y$ to be the group $\Br Y = \H^2(Y,\Gm)$.  Consider the diagonal embedding of $Y(k)$ into $Y(\Adele_k)$.  Manin observed that the image of this embedding must be contained in the subset
\[
Y(\Adele_k)^{\Br} = \big\{ (x_v) \in Y(\Adele_k) \mid \sum_v \inv_v \cA(x_v) = 0 \text{ for all } \cA \in \Br Y \big\}.
\]
Here $\inv_v \colon \Br k_v \to \QZ$ is the local invariant map, and the sum is over all places of $k$.
If $Y(\Adele_k)^{\Br}$ is empty, then $Y(k)$ is also empty.  If in addition $Y(\Adele_k)$ is non-empty, then we say that there is a \emph{Brauer--Manin obstruction to the Hasse principle} on $Y$.  Whenever $Y(\Adele_k)^{\Br}$ differs from $Y(\Adele_k)$, we say that there is a \emph{Brauer--Manin obstruction to weak approximation} on $Y$.  By restricting $\cA$ to lie in a subset $B \subset \Br Y$, we can also talk about an obstruction coming from $B$.  Since the quantity $\sum_v \inv_v \cA(x_v)$ only depends on the class of $\cA$ modulo $\Br k$, we can also talk about the obstruction coming from a given subset of $\Br Y / \Br k$.  The set $Y(\Adele_k)^{\Br}$ is effectively computable in many situations of interest.  We say that \emph{the Brauer--Manin obstruction is the only obstruction to the Hasse principle} for a class $\cC$ of varieties if, for all varieties $Y \in \cC$, the implication $Y(\Adele_k)^{\Br} \neq \emptyset \Rightarrow Y(k) \neq \emptyset$ holds.
There are varieties for which the Brauer--Manin obstruction is not the only one: the first example was given by Skorobogatov~\cite{Skorobogatov:IM-1999}.
But it has been conjectured by Colliot-Th\'el\`ene~\cite[p.~319]{CT:AFST-1992} that the Brauer--Manin obstruction is the only obstruction to the Hasse principle for geometrically rational varieties.  This conjecture is far from being proved apart from in a few very specific cases; however, it supports the idea that studying the Brauer--Manin obstruction is a useful proxy for studying the existence of rational points on varieties.

Experience with examples of Brauer--Manin obstructions seems to show that it is much easier to construct counterexamples to weak approximation than to construct counterexamples to the Hasse principle.  Indeed, given a variety $Y$ with non-trivial Brauer group, one ``usually'' expects there to be an obstruction to weak approximation on $Y$ but no obstruction to the Hasse principle.  This article can be seen as an attempt to make these empirical observations precise.

A useful example to keep in mind is that of diagonal cubic surfaces.  Colliot-Th\'el\`ene, Kanevsky and Sansuc proved the following:
\begin{theorem*}[\cite{CTKS}, Corollary to Proposition~2]
Let $V$ be the cubic surface over $\QQ$ given by the equation
\[
ax^3 + by^3 + cz^3 + dt^3 = 0,
\]
where $a,b,c,d$ are non-zero cube-free integers.  If there is a prime number dividing precisely one of $a,b,c,d$, then there is no Brauer--Manin obstruction to the Hasse principle on $V$.
\end{theorem*}
One checks that the condition holds for 100\% of $(a:b:c:d) \in \PP^3(\QQ)$.  The method of proof, at least for generic $V$ satisfying the condition of the theorem, is as follows. The authors firstly calculate that $\Br V / \Br \QQ \isom \H^1(\QQ, \Pic\bar{V})$ has order $3$.  Letting $\cA \in \Br V$ be a generator for $\Br V / \Br \QQ$, and taking $p \neq 3$ to be the prime number of the theorem, they show that the map $V(\QQ_p) \to \Br \QQ_p$ given by evaluating $\cA$ surjects onto $\Br \QQ_p[3]$.  From there, it is easy to deduce the absence of a Brauer--Manin obstruction to the Hasse principle, and indeed the presence of an obstruction to weak approximation.  (The case $p=3$ is rather different and is treated separately; but 100\% of varieties in the family satisfy the condition with $p \neq 3$.)

In~\cite{BBL:wa}, we studied weak approximation in families of varieties.  Theorem~1.6 of that article considers a family of varieties defined by a morphism $X \to \PP^m$, and gives sufficient conditions under which 100\% of locally soluble members of the family have a Brauer--Manin obstruction to weak approximation.  In this article, we extend the method of~\cite{BBL:wa} to prove a corresponding result for the Hasse principle: that, in a suitable family of varieties, 100\% have no Brauer--Manin obstruction to the Hasse principle.  
The method used in the present article is, of necessity, strictly stronger.
To give an obstruction to weak approximation, one only has to consider a single element of the Brauer group.  To prove the absence of an obstruction to the Hasse principle, it is not enough to show that no single element gives an obstruction; one must consider the whole Brauer group simultaneously.
The main challenge of this article is to extend the technical results of~\cite{BBL:wa} from single elements to the whole of the Brauer group.

Denote by $\kbar$ an algebraic closure of the number field $k$.  The \emph{algebraic Brauer group} of a variety $Y$ is the subgroup $\Br_1 Y = \ker ( \Br  Y \to \Br \bar{Y} )$, where $\bar{Y}$ denotes the base change of $Y$ to $\kbar$.  Thus $\Br_1 Y$ consists of those elements that become trivial over $\kbar$.  In this article we are principally interested in algebraic Brauer groups (though the main technical results, in Section~\ref{sec:surj}, apply more generally).  The reason for this is that we can use a result of Harari~\cite{Harari:DMJ-1994} to show that the algebraic Brauer group is the same for ``most'' varieties in a family.

We would like to construct Brauer group elements on the varieties in a family by restricting from the generic fibre.  However, as was shown by Uematsu~\cite{Uematsu:QJM-2014}, this is not always possible: in the family of diagonal cubic surfaces, a general element has non-trivial algebraic Brauer group, whereas the generic fibre has trivial algebraic Brauer group.
To avoid this problem, we make use of the isomorphism $\Br_1 Y / \Br k \isom \H^1(k, \Pic \bar{Y})$ which exists whenever $Y$ is a smooth, projective, geometrically irreducible variety over a number field $k$.  Instead of specialising elements of the Brauer group of the generic fibre, we specialise elements of the corresponding cohomology group, which is $\H^1(K, \Pic X_{\bar{\eta}})$ in the notation of Theorem~\ref{thm} below.

In order to count points, we need a notion of height.    
For each finite place $v$ of the number field $k$, let $\|\cdot\|_v$ be the corresponding normalised absolute value on $k_v$, that is, the absolute value satisfying $\| \pi \|_v = 1/q$ where $\pi$ is a uniformiser in $k_v$ and $q$ is the size of the residue field at $v$.
If $v$ is a real place, then define $\|\cdot\|_v$ to be the usual real absolute value on $k_v$; is $v$ is a complex place, define $\|\cdot\|_v$ to be the square of the usual complex absolute value.
For a point $P \in \PP^m(k)$, we denote by $H(P)$ the exponential height of $P$, defined as follows:
choose coordinates $[x_0 : \dotsb : x_n]$ for $P$; then
\[
H(P) = \prod_v \max \{ \|x_0\|_v, \dotsc, \|x_n\|_v \},
\]
where the product is over all places $v$ of $k$.

\begin{theorem}\label{thm}
Let $k$ be a number field and let $\pi \colon X \to \PP^m_k$ be a flat, surjective morphism of finite type, with $X$ smooth, projective and geometrically integral.  Let $\eta \colon \Spec K \to \PP^m_k$ denote the generic point and $\bar{\eta} \colon \Spec \bar{K} \to \PP^m_k$ a geometric point above $\eta$.  Suppose that the geometric generic fibre $X_{\bar{\eta}}$ is connected and has torsion-free Picard group.  Denote by $\Xbar$ the base change of $X$ to an algebraic closure $\kbar$ of $k$.  Assume the following hypotheses:
\begin{enumerate}
\item\label{locsol} $X(\Adele_k) \neq \emptyset$;
\item\label{codim1} the fibre of $\pi$ at each codimension-1 point of $\PP^m_k$ is geometrically integral;
\item\label{codim2} the fibre of $\pi$ at each codimension-2 point of $\PP^m_k$ has a geometrically reduced component;
\item\label{h1zero} $\H^1(k, \Pic \bar{X})=0$;
\item\label{brzero} $\Br \bar{X}=0$;
\item\label{h2inj} $\H^2(k, \Pic \PP^m_{\bar{k}}) \to \H^2(k, \Pic \bar{X})$ is injective.
\end{enumerate}
Denote by $U \subset \PP^m_k$ the open subset over which the fibre of $\pi$ are smooth.
Then, for 100\% of rational points $P \in U(k) \cap \pi(X(\Adele_k))$, the image of $\H^1(K, \Pic X_{\bar{\eta}})$ in $\Br X_P/\Br k$ gives no Brauer--Manin obstruction to the Hasse principle on $X_P$.
\end{theorem}

The conclusion of the theorem should be interpreted as follows: let $T$ denote the set $U(k) \cap \pi(X(\Adele_k))$, and let $T' \subset T$ be the subset of those $P \in T$ for which the image of $\H^1(K, \Pic X_{\bar{\eta}})$ in $\Br X_P/\Br k$ gives no Brauer--Manin obstruction to the Hasse principle on $X_P$.  Then
\[
\limsup_{B \to \infty} \frac{ \#\{ P \in (T \setminus T') \mid H(P) \le B \}}{ \#\{ P \in T \mid H(P) \le B \}} = 0.
\]

This theorem should be compared to Theorem~1.6 of~\cite{BBL:wa}, which has very similar hypotheses.  There, the conclusion was that 100\% of smooth fibres fail weak approximation; here, we prove the related but harder result that a Brauer--Manin obstruction vanishes.
The only differences in the hypotheses are as follows: in~\cite[Theorem~1.6]{BBL:wa} there was an assumption that either $\H^1(K,\Pic X_{\bar{\eta}})$ or $\Br X_\eta/\Br K$ be non-zero; but our Theorem~\ref{thm} is trivially true if $\H^1(K,\Pic X_{\bar{\eta}})$ vanishes.  On the other hand, Theorem~\ref{thm} demands that $X_{\bar{\eta}}$ have torsion-free Picard group; the conclusion of~\cite[Theorem~1.6]{BBL:wa} is automatically true when that condition does not hold, so the condition was not needed there.

Note that, under the hypotheses of Theorem~\ref{thm}, a positive proportion of the fibres $X_P$ are everywhere locally soluble: see~\cite[Theorem~1.3]{BBL:wa}.

\begin{remark}
The requirement that $\pi$ be flat is rather unhelpful for applications.  Families of interest often arise over an open base, which is then compactified.  The compactification may have singularities; resolving these singularities gives a regular family over a compact base, but possibly at the expense of flatness.

However, some such condition is needed for our proof to work: suppose that there were an element of $\Br X_\eta$, ramified on $X$ at a divisor having image of codimension $\ge 2$ in $\PP^m$.  
The approach followed in this article would be unable to show that such a Brauer element does not obstruct the Hasse principle.

The flatness condition is ultimately used in the proof of~\cite[Proposition~5.6]{BBL:wa}, which depends on the condition that, for any closed subset $S \subset \PP^m$ of codimension $c \le 3$, the inverse image $\pi^{-1}(S)$ also has codimension $c$ in $X$.  The proof therefore remains valid under the slightly weaker assumption that $\pi$ is flat outside a closed subset of codimension $\ge 3$ in $X$.
\end{remark}

\begin{remark}
In view of condition~(\ref{codim1}), there is a short exact sequence
\[
\Pic \PP^m_\kbar \to \Pic \Xbar \to \Pic \Xbar_\eta \to 0
\]
of $\Gal(\kbar/k)$-modules.  (Here $\Xbar_\eta$ is the generic fibre of $\Xbar \to \PP^m_\kbar$, not to be confused with $X_{\bar{\eta}}$.)  Moreover, the first map is injective, since the pull-back of a hyperplane in $\PP^m_\kbar$ has positive intersection number with a non-vertical curve in $X$.
Looking at the resulting long exact sequence in cohomology shows that conditions~(\ref{h1zero}) and~(\ref{h2inj}) are equivalent to the single condition $\H^1(k, \Pic \Xbar_\eta) = 0$.

One situation in which $\H^1(k, \Pic \Xbar_\eta) = 0$ is not satisfied is in the case of a constant family.  If $Y$ is a variety over $k$ having an obstruction to the Hasse principle given by an algebraic Brauer class, then taking $X=Y \times \PP^m_k$ gives an example failing this condition and for which the conclusion of the theorem is completely false.
\end{remark}

A specific case in which conditions~(\ref{h1zero}), (\ref{brzero}) and~(\ref{h2inj}) are satisfied is when $X$ is a non-constant complete intersection of dimension $\ge 3$ in $\PP^r \times \PP^m \to \PP^m$: see~\cite[Proposition~5.17]{BBL:wa}.  In this case, flatness is automatic as soon as the fibres of $X$ all have the same dimension: see~\cite[Chapter~I, Remarks~2.6]{Milne:EC}.

A theorem of Harari~\cite{Harari:DMJ-1994} shows that, under the additional assumption that the Brauer group of the geometric generic fibre $X_{\bar{\eta}}$ is trivial, the image under specialisation of $\H^1(K, \Pic X_{\bar{\eta}})$ is the whole of $\Br X_P/\Br k$ for most points $P$.  (Here ``most'' means that $P$ should lie in a Hilbertian subset of $\PP^m(k)$.)  This is the case, for example, when $X_\eta$ is a geometrically rational variety.  Combining this with Theorem~\ref{thm}, we obtain the following.

\begin{theorem}\label{thm2}
Under the conditions of Theorem~\ref{thm}, suppose moreover that $\Br X_{\bar{\eta}}$ is trivial.  Then, for 100\% of rational points $P \in U(k) \cap \pi(X(\Adele_k))$, there is no Brauer--Manin obstruction to the Hasse principle on $X_P$.
\end{theorem}

The condition that $X_{\bar{\eta}}$ have torsion-free Picard group is a necessary one.  Indeed, Bhargava~\cite{Bhargava:hasse} has shown that a positive proportion of plane cubic curves fail the Hasse principle; if one assumes finiteness of Tate--Shafarevich groups, then these failures are all due to a Brauer--Manin obstruction~\cite[Theorem~6.2.3]{Skorobogatov:TRP}.

Although Theorem~\ref{thm2} says that Brauer--Manin obstructions to the Hasse principle are in some sense ``rare'', this is strictly in an arithmetic sense and not in a geometric one: Jahnel and Schindler~\cite{JS:AIF-2017} have constructed a set of del Pezzo surfaces of degree 4, all having a Brauer--Manin obstruction to the Hasse principle, that are Zariski-dense in the family of all del Pezzo surfaces of degree 4.  (That example does not quite fit into our setting, but only because the base scheme of the family is not a projective space but rather a Grassmannian.)

\begin{remark}
It is natural to ask whether statements such as Theorem~\ref{thm} should be expected when the base of the family is a variety other than projective space.  For example, one could ask what happens when the base is a Fano variety, and we count points with respect to the height given by an ample divisor.
Indeed, the geometric arguments both in this article and in~\cite{BBL:wa} apply in a broader setting.
However, there are barriers relating purely to the base variety, regardless of the particular family of varieties.

Our quantitative statements rely fundamentally on properties of the distribution of rational points on the base, which cannot be expected to hold unaltered for bases other than projective space.
In particular, the most na\"ive generalisation of Proposition~4.3 of~\cite{BBL:wa} is no longer true, as we now explain.  
Such a generalisation might be expected to state, roughly: given a variety $Y$ over a number field $k$, a divisor $D$ on $Y$ and an open subset $\sV$ of a model of $D$, then 100\% of the points of $Y(k)$ land in $\sV$ when reduced modulo some prime of $k$.
For a simple counterexample, take $Y$ to be a smooth cubic surface over $\QQ$ containing three rational, coplanar lines.  
Heath-Brown~\cite{HB:AA-1997} has proved that 100\% of the rational points on $Y$, ordered by height, lie on the lines.  (Indeed, Manin's conjecture would imply that such a phenomenon happens in much greater generality.)
Let $D \subset Y$ be a curve not contained in the union of the 27 lines on $Y$.
Choose a model for $Y$ over $\ZZ$, and let $\sV$ be the Zariski closure of $D$ minus the Zariski closure of the 27 lines;
then 100\% of the rational points of $Y$ lie on the lines, and so fail to land in $\sV$ modulo any prime.

The possibility remains open that some form of the result could hold after passing to a dense open subset of the base.
See~\cite{BL:sieving} for an approach to this and related questions.
\end{remark}

Throughout this article, all cohomology is \'etale cohomology, and
we use the notation $\Br S$ to denote the cohomological Brauer group $\H^2(S,\Gm)$ of a scheme $S$.
If $A$ is an Abelian group, then $A[n]$ and $A/n$ denote the kernel and cokernel, respectively, of multiplication by $n$, and $\dual{A}$ denotes the dual $\Hom(A,\QZ)$.  We will often use the fact that, for any scheme $S$, the natural injection $\Zn \to \QZ$ identifies $\H^1(S,\Zn)$ with the $n$-torsion in $\H^1(S,\QZ)$.

I thank Tim Browning, Daniel Loughran, Bas Edixhoven, Efthymios Sofos and several referees for their helpful comments and discussions relating to this article.

\section{Outline of the proof}

Let $k$ be a number field, and let $Y$ be a smooth, proper, geometrically irreducible variety over $k$.  Assume that $Y$ is everywhere locally soluble, that is, $Y(\Adele_k)$ is non-empty, and let $\cA \in \Br Y$ be a class in the Brauer group.  A much-used strategy for proving that $\cA$ gives no obstruction to the Hasse principle on $Y$ is the following.  Suppose that $\cA$ has order $n$ in $\Br Y$, and that there exists a finite place $v$ of $k$ such that the evaluation map $Y(k_v) \to \Br k_v[n]$, sending a point $P \in Y(k_v)$ to the evaluation $\cA(P) \in \Br k_v[n]$, is surjective.  Then $\cA$ gives no obstruction to the Hasse principle.  Indeed, any adelic point in $Y(\Adele_k)$ can be altered at the place $v$ to give an adelic point orthogonal to $\cA$.

As stated, this strategy is unsatisfactory because the surjectivity condition can be destroyed by changing $\cA$ by a constant algebra, that is, an element of the image of $\Br k \to \Br Y$.  Adding a constant algebra of arbitrarily large order, we can make the order of $\cA$ in $\Br Y$ arbitrarily large, yet the image of the evaluation map is simply translated and so stays the same size.  To compensate for this phenomenon, we can fix a base point $P \in Y(k_v)$ and consider the modified evaluation map $Y(k_v) \to \Br k_v$ defined by $Q \mapsto \cA(Q) - \cA(P)$.  This evaluation map only depends on $\cA$ modulo constant algebras.  As before, if the class of $\cA$ in $\Br Y / \Br k$ has order $n$ and the evaluation map surjects onto $\Br k_v[n]$, one can show that $\cA$ gives no obstruction to the Hasse principle.

In this article, we need to generalise the above strategy in two ways: to more than one algebra, and to more than one place of $k$.  The need to understand more than one algebra is clear: if $\Br Y/\Br k$ is not cyclic, then showing that there is no Brauer--Manin obstruction to the Hasse principle on $Y$ is strictly stronger than showing that no individual element of $\Br Y$ obstructs the Hasse principle.  As before, fix a place $v$ of $k$ and a point $P \in Y(k_v)$.  We consider the pairing
\begin{equation}\label{eq:brpairing}
\Br Y / \Br k \times Y(k_v) \to \QZ
\end{equation}
defined by
\[
(\cA,Q) \mapsto \inv_v \cA(Q) - \inv_v \cA(P).
\]
In~\cite[Definition~7.1]{Bright:bad}, a subgroup $A \subset \Br Y / \Br k$ was defined to be \emph{prolific} at $v$ if the map $Y(k_v) \to \Hom(A,\QZ)$ induced by the pairing~\eqref{eq:brpairing} is surjective.  One easily proves that, if there exists a place of $k$ where $A$ is prolific, then $A$ gives no Brauer--Manin obstruction to the Hasse principle.

We will also need to generalise this argument to several places of $k$, as illustrated by the following example.  Suppose that $\cA \in \Br Y$ is an algebra having order $6$ in $\Br Y / \Br k$.  Let $v$ and $w$ be two places of $k$ and fix base points $P_v \in Y(k_v)$ and $P_w \in Y(k_w)$.  Suppose that
\begin{itemize}
\item the evaluation map $Q_v \mapsto \big( \inv_v \cA(Q_v) - \inv_v \cA(P_v) \big)$ maps $Y(k_v)$ onto the subgroup $\{0,\frac{1}{2}\} \subset \QZ$; and
\item the evaluation map $Q_w \mapsto \big( \inv_w \cA(Q_w) - \inv_w \cA(P_w) \big)$ maps $Y(k_w)$ onto the subgroup $\{0,\frac{1}{3},\frac{2}{3}\} \subset \QZ$.
\end{itemize}
Then, by choosing points at both places $v$ and $w$, we can produce any invariant in $(\frac{1}{6}\ZZ)/\ZZ$.  As before, we conclude that $\cA$ gives no obstruction to the Hasse principle on $Y$.  With this example in mind, we make the following definition.

Let $S$ be a finite set of places of $k$, and fix a point $P_v \in Y(k_v)$ for each $v \in S$.  There is an evaluation pairing
\begin{equation}\label{eq:evalS}
(\Br Y / \Br k) \times \prod_{v \in S} Y(k_v) \to \QZ
\end{equation}
defined by
\[
(\mathcal{A},(Q_v)_{v \in S}) \mapsto \sum_{v \in S} (\inv_v \mathcal{A}(Q_v) - \inv_v \mathcal{A}(P_v)).
\]
The following definition extends that of~\cite[Definition~7.1]{Bright:bad}.
\begin{definition}
A subgroup $A$ of $\Br Y / \Br k$ is \emph{prolific} at $S$ if the map
\[
\prod_{v \in S} Y(k_v) \to \Hom(A,\QZ)
\]
induced by the above pairing is surjective.
\end{definition}
The definition is easily seen to be independent of the choice of the $P_v$.  A straightforward extension of~\cite[Proposition~7.3]{Bright:bad} shows that, if $A \subset \Br Y/\Br k$ is prolific at any set $S$ of places, then $A$ gives no Brauer--Manin obstruction to the existence of rational points on $Y$.

To prove Theorem~\ref{thm}, then, it suffices to show that 100\% of the fibres $X_P$ admit a set $S$ of primes at which the relevant subgroup of $\Br X_P/\Br k$ is prolific.  
It was shown in~\cite{Bright:bad} that, at least for $v$ sufficiently large, the evaluation map at $v$ corresponding to an element $\cA$ of $\Br_1 X_P$ only depends on the residue of $\cA$ at $v$.  To control these residues, we use the philosophy of~\cite{BBL:wa}.
Specifically, we generalise the following proposition.  The notation is that of Theorem~\ref{thm}; the relative residue map $\rho_d$ will be defined in Section~\ref{sec:res}.

\begin{proposition*}[\cite{BBL:wa}, Proposition~4.2]
Let $\alpha\in  \H^1(K,\Pic X_{\bar{\eta}})$
be a class of order $n >1$ and suppose that there is a point $d$ of codimension $1$ in $\PP^m$ such that the relative residue 
\[\rho_d(\alpha) \in \H^0(\ff{d},\H^1(X^\sm_{\bar{d}},\Zn))\] also has order $n$.
Let $\sX \to \PP^m_\fo$ be a model of $\pi$ over $\fo$, the ring of integers of $k$.
Denote by $\sD$ the Zariski closure of $d$ in $\PP^m_\fo$.  Then there is a dense open subset $\sV \subset \sD$ with the following property:
\begin{quote}
Let $P$ be a point of $U(k)$ such that $X_P$ is everywhere locally soluble, and suppose that the Zariski closure of $P$ in $\PP^m_\fo$ meets $\sV$ transversely at a closed point $s$.  Let $\p$ be the prime of $\fo$ over which $s$ lies.  Then
the algebra $\spec_P(\alpha) \in \Br X_P$ is prolific at $\fp$.
\end{quote}
\end{proposition*}
Furthermore, \cite[Proposition~4.3]{BBL:wa} shows that 100\% of points $P \in U(k)$ satisfy the required property, so that we obtain 100\% of $P$ for which $\spec_P(\alpha) \in \Br X_P$ is prolific at some place $\fp$.  This is sufficient to show that there is an obstruction to weak approximation on $X_P$.

However, this does not suffice two prove the absence of a Brauer--Manin obstruction to the Hasse principle, for the two reasons outlined at the beginning of this section.  Firstly, we must study more than one algebra: indeed, we have to study simultaneously all algebras arising by specialising elements of $\H^1(K, \Pic X_{\bar{\eta}})$.  Secondly, an algebra of composite order $n$ cannot in general be expected to be prolific at a single place, so we must combine information from several places.  This leads us to Proposition~\ref{prop:main} and Corollary~\ref{cor:prolific}, substantially generalising~\cite[Proposition~4.2]{BBL:wa}, the proofs of which make up the bulk of this article.  From there, we prove Theorem~\ref{thm}.


The reason that Theorem~\ref{thm} does not prove the absence of any Brauer--Manin obstruction on 100\% of $X_P$ is that individual $X_P$ may have Brauer group elements other than those obtained by specialising $\H^1(K, \Pic X_{\bar{\eta}})$.  However, if we further assume that $\Br X_{\bar{\eta}}$ is trivial, then a result of Harari~\cite{Harari:DMJ-1994} on specialisation of the algebraic Brauer group states the following: for $P$ lying in a Hilbertian subset of $\PP^m(k)$, specialisation gives an isomorphism $\H^1(K,\Pic X_{\bar{\eta}}) \isom \H^1(k, \Pic \bar{X}_P)$.  The hypothesis also implies that $\Br \bar{X}_P$ vanishes, and so the whole of $\Br X_P$ is algebraic; thus Theorem~\ref{thm} applies to the whole of $\Br X_P$ and shows that there is indeed no Brauer--Manin obstruction on $X_P$.  This is the content of Theorem~\ref{thm2}.

In Section~\ref{sec:res}, we recall the definition and properties of relative residue maps for $\H^1(K, \Pic X_{\bar{\eta}})$ from~\cite{BBL:wa}.  In Section~\ref{sec:surj}, we prove the key technical result (Proposition~\ref{prop:main}) showing that, by choosing $P$ suitably, we can ensure that the specialisation to $X_P$ of $\H^1(K, \Pic X_{\bar{\eta}})$ is prolific.  In Section~\ref{sec:sieve}, we apply a result of sieve theory from~\cite{BBL:wa} to prove that 100\% of $P \in U(k)$ satisfy all the desired conditions, proving Theorem~\ref{thm}.  In Section~\ref{sec:spec} we apply Harari's result to show that, for 100\% of $P \in \PP^m(k)$, the specialisation of $\H^1(K, \Pic X_{\bar{\eta}})$ is the whole of $\Br_1 X_P / \Br k$, thus proving Theorem~\ref{thm2}.

\section{Residue maps}\label{sec:res}

Let $Y$ be a regular scheme, and $Z$ a regular prime divisor on $Y$.  Suppose for now that the residue characteristics of $Y$ are all zero.  There are various equivalent ways to define a residue map $\partial_Z \colon \Br (Y \setminus Z) \to \H^1(Z,\QZ)$ fitting into an exact sequence
\[
0 \to \Br Y \to \Br (Y \setminus Z) \xrightarrow{\partial_Z} \H^1(Z,\QZ).
\]
One definition is to use the Kummer sequence, which gives a short exact sequence
\[
0 \to \Pic (Y \setminus Z)/ n \to \H^2(Y \setminus Z, \mmu_n) \to \Br (Y \setminus Z)[n] \to 0.
\]
There is then a residue map, which we will also denote by $\partial_Z$, from $\H^2(Y \setminus Z,\mmu_n)$ to $\H^1(Z,\Zn)$, arising from the Gysin sequence for the pair $Z \subset Y$; see, for example, \cite[Section~2]{Bright:bad} for details.  This residue map vanishes on the image of $\Pic (Y \setminus Z)/n$ and so induces a residue map on $\Br(Y \setminus Z)[n]$.  Combining these for all $n$ defines the residue map on the torsion group $\Br(Y \setminus Z)$.  Without the assumption on the residue characteristics of $Y$, the same works for any $n$ invertible on $Y$.

In the notation of Theorem~\ref{thm}, let $Z \subset X$ be a vertical prime divisor; applying the above construction to the local ring of $X$ at $Z$ allows us to construct residue maps $\Br X_\eta \to \H^1(\ff{Z},\QZ)$.  However, we are interested in classes in $\H^1(K, \Pic X_{\bar{\eta}})$, and it was shown by Uematsu~\cite{Uematsu:QJM-2014} that, in contrast to the case of varieties over a number field, the natural map $(\Br_1 X_\eta / \Br K) \to \H^1(K, \Pic X_{\bar{\eta}})$ is not necessarily surjective.  To address this problem, in~\cite{BBL:wa} we defined residue maps on $\H^1(K, \Pic X_{\bar{\eta}})$ extending those on $\Br_1 X_\eta$.  We now briefly recall the definition and properties of those residue maps.

As in the usual case, it is easier to work with finite coefficients, so we first relate $\H^1(K,\Pic X_{\bar{\eta}})$ to some cohomology groups with finite coefficients.  Let $n$ be a positive integer, and suppose that $\Pic X_{\bar{\eta}}$ has no $n$-torsion.  The multiplication-by-$n$ map gives a short exact sequence in cohomology
\[
0 \to \H^0(K, \Pic X_{\bar{\eta}})/n \to \H^0(K, \Pic X_{\bar{\eta}}/n) \to \H^1(K, \Pic X_{\bar{\eta}})[n] \to 0.
\]
On the other hand, the Kummer sequence on $X_{\bar{\eta}}$ gives an injective homomorphism $\Pic X_{\bar{\eta}}/n \to \H^2(X_{\bar{\eta}},\mmu_n)$.  So, given a class in $\H^1(K, \Pic X_{\bar{\eta}})[n]$, we may lift it to $\H^0(K, \Pic X_{\bar{\eta}}/n)$ and then push forward to $\H^0(K, \H^2(X_{\bar{\eta}},\mmu_n))$.  This gives a well-defined homomorphism
\begin{equation}\label{eq:h1h2}
\H^1(K, \Pic X_{\bar{\eta}})[n] \to \H^0(K, \H^2(X_{\bar{\eta}}, \mmu_n)) / \H^0(K, \Pic X_{\bar{\eta}})
\end{equation}
where the right-hand side, by abuse of notation, means the cokernel of the homomorphism $\H^0(K, \Pic X_{\bar{\eta}})  \to \H^0(K, \H^2(X_{\bar{\eta}}, \mmu_n))$ induced by the cycle class map.

To define the new residue maps, we again put ourselves in a slightly more general situation.  Let $f \colon Y \to B$ be a morphism of regular schemes; let $D \subset B$ be a regular prime divisor of $B$, and suppose that $Z = f^{-1}(D)$ is a regular prime divisor on $Y$.  Let $n$ be a positive integer invertible on $B$.  The Gysin sequence gives a map $(\mmu_n)_{Y \setminus Z} \to (\Zn)_Z[-1]$; applying $\R^2 f_*$ and taking global sections produces a homomorphism 
\[
\rho_D \colon \H^0(B \setminus D, \R^2 f_* \mmu_n) \to \H^0(D, \R^1 (f_Z)_* \Zn),
\]
called the \emph{relative residue map} at $D$.

In the situation of Theorem~\ref{thm}, let $d$ be a point of dimension $1$ in $\PP^m_k$, that is, $d$ is the generic point of a prime divisor on $\PP^m_k$.  Applying the relative residue construction to the local ring at $d$ gives a relative residue map
\[
\rho_d \colon \H^0(K, \H^2(X_{\bar{\eta}}, \mmu_n)) \to \H^0(\ff{d}, \H^1(X_{\bar{d}}^\sm,\Zn))
\]
where $\bar{d}$ is a geometric point lying above $d$, and $X_{\bar{d}}^\sm$ is the open subscheme of $X_{\bar{d}}$ on which $\pi$ is smooth.  Under conditions~(\ref{codim1}) and~(\ref{codim2}) of Theorem~\ref{thm}, Proposition~5.10 of~\cite{BBL:wa} shows that these residue maps fit into an exact sequence
\begin{multline*}
0 \to \H^0(B, \R^2 \pi_* \mmu_n) \to \H^0(K, \H^2(X_{\bar{\eta}}, \mmu_n)) \\\xrightarrow{\oplus \rho_d} 
\bigoplus_{d \in (\PP^m_k)^{(1)}} \H^0(\ff{d}, \H^1(X_{\bar{d}}^\sm,\Zn)).
\end{multline*}
Lemma~5.12 of~\cite{BBL:wa} shows that the image of $\H^0(K, \Pic X_{\bar{\eta}})$ in $\H^0(K, \H^2(X_{\bar{\eta}},\mmu_n))$ is contained in the kernel of every $\rho_d$, so that the homomorphism~\eqref{eq:h1h2} induces for each $d$ a well-defined residue map
\[
\H^1(K, \Pic X_{\bar{\eta}})[n] \to \H^0(\ff{d}, \H^1(X_{\bar{d}}^\sm, \Zn)), 
\]
which we will also denote $\rho_d$.  

\section{Surjectivity of evaluation maps}\label{sec:surj}

In this section we prove the main technical ingredient, Proposition~\ref{prop:main}, in the proof of Theorem~\ref{thm}.  Our Proposition~\ref{prop:main} is a generalisation of Proposition~4.2 of~\cite{BBL:wa}, and its proof follows the structure of the proof found in Section~5.6 of~\cite{BBL:wa}, though with substantial technical additions.  In this section, there is no need to assume that the base scheme is $\PP^m$, so instead of working in the context of Theorem~\ref{thm} we put ourselves in the following more general situation:
\begin{quote}
($\star$) $k$ is a number field with ring of integers $\fo$.  $\pi \colon X \to B$ is a flat, surjective morphism of smooth varieties over $k$.  If $K$ is the function field of $B$ and $\eta \colon \Spec K \to B$ is the generic point, then the generic fibre $X_\eta$ is smooth, proper and geometrically connected with torsion-free geometric Picard group.  The fibre of $\pi$ at each codimension-1 point of $B$ is geometrically integral.  The fibre of $\pi$ at each codimension-2 point of $B$ has a geometrically reduced component.
\end{quote}
Most of the time we will work with $\H^0(K, \H^2(X_{\bar{\eta}},\mmu_n))$, only later applying the results to $\H^1(K, \Pic X_{\bar{\eta}})$.  We begin by saying what it means to specialise an element of $\H^0(K, \H^2(X_{\bar{\eta}},\mmu_n))$ at a point.

Let $U \subset B$ be the dense open subset above which the fibres of $\pi$ are smooth.  
There is a natural map $\eta^* \colon \H^0(U, \R^2 \pi_* \mmu_n) \to \H^0(K, \H^2(X_{\bar{\eta}},\mmu_n))$ given by base change by $\eta \colon \Spec K \to U$;
Lemma~5.19 of~\cite{BBL:wa} states that it is an isomorphism.
Let $P$ be a point of $U(k)$.  Suppose that the $k$-variety $X_P$ is everywhere locally soluble; then Lemma~5.20 of~\cite{BBL:wa} shows that the natural map
\[
\H^2(X_P,\mmu_n) / \H^2(k,\mmu_n) \to \H^0(k, \H^2(\bar{X}_P,\mmu_n)),
\]
arising from the Hochschild--Serre spectral sequence, is an isomorphism.
We thus have a sequence of homomorphisms
\[
\begin{CD}
\H^0(K, \H^2(X_{\bar{\eta}}, \mmu_n)) \\
@A{\eta^*}A{\sim}A\\
\H^0(U,\R^2 \pi_* \mmu_n) \\
@V{P^*}VV \\
\H^0(k, \H^2(\bar{X}_P, \mmu_n)) \\
@AA{\sim}A \\
\H^2(X_P,\mmu_n)/\H^2(k,\mmu_n) \\
@VVV \\
\Br X_P[n]/\Br k[n]
\end{CD}
\]
in which both upward-pointing arrows are isomorphisms, allowing us to define, for a class $\alpha \in \H^0(K, \H^2(X_{\bar{\eta}}, \mmu_n))$, the specialisation $\spec_P(\alpha) \in \Br X_P[n] / \Br k[n]$.
Fix a place $v$ of $k$, and a base point $x_0 \in X_P(k_v)$.  Then we define an evaluation pairing
\begin{equation}\label{eq:pairbr}
\H^0(K, \H^2(X_{\bar{\eta}}, \mmu_n)) \times X_P(k_v) \to \Br k_v[n]
\end{equation}
by
\[
(\alpha, x) \mapsto \spec_P(\alpha)(x) - \spec_P(\alpha)(x_0).
\]
In the expression above, we are implicitly lifting $\spec_P(\alpha)$ to $\Br X_p[n]$, but the resulting value is independent of the choice of lift.  Composing with the local invariant map $\inv_v \colon \Br k_v \to \QZ$ gives a pairing
\begin{equation}\label{eq:pairqz}
\H^0(K, \H^2(X_{\bar{\eta}}, \mmu_n)) \times X_P(k_v) \to \QZ.
\end{equation}
This is simply the pairing obtained by composing the homomorphism $\spec_P$ with the pairing~\eqref{eq:brpairing}.
In both cases, the effect of changing the base point $x_0$ is to translate the value of the pairing by a constant.  

By a \emph{model} of $\pi$ over $\fo$ we mean an integral scheme $\sB$, separated and of finite type over $\fo$, satisfying $\sB \times_{\fo} k = B$, together with an $\fo$-morphism $\sX \to \sB$, separated and of finite type, extending $\pi$.

To simplify the statement of the following proposition, we make some more definitions.  The first is a generalisation of what it means for the specialisation of a certain subgroup to be prolific on a given fibre.

\begin{definition}
Let $A \subset \H^0(K, \H^2(X_{\bar{\eta}}, \mmu_n))$ be a finite subgroup, let $D \subset B^{(1)}$ be a finite set of points of codimension $1$ in $B$, and let $P \in U(k)$ be a point such that $X_P$ is everywhere locally soluble.  Let $S$ be a finite set of places of $\fo$.
We say that $A$ is \emph{$D$-prolific at $S$ on the fibre $X_P$} if the image of the map 
\[
\prod_{v \in S} X_P(k_v) \to \dual{A}
\]
defined by the sum of the pairings~\eqref{eq:pairqz} contains the image of the map 
\[
\bigoplus_{d \in D}\dual{\H^0(\ff{d},\H^1(X_{\bar{d}}^\sm, \Zn))} \to \dual{A}
\]
induced by $\oplus_{d \in D} \rho_d$.
\end{definition}

In particular, if $\oplus \rho_d \colon A \to \bigoplus_{d \in D} \H^0(\ff{d},\H^1(X_{\bar{d}}^\sm, \Zn))$ happens to be injective, then saying that $A$ is $D$-prolific at $S$ on the fibre $X_P$ is the same as saying that $\spec_P(A) \subset (\Br X_P[n]/\Br k[n])$ is prolific at $S$.

The second definition encapsulates what it means for a point $P \in U(k)$ to land in a given divisor modulo some prime, in a suitably generic way. 

\begin{definition}
Fix a model $\sX \to \sB$ of $\pi$.  Let $d \in B$ be a point of codimension $1$ not contained in $U$, let $\sD$ be the Zariski closure of $d$ in $\sB$ (which is a prime divisor on $\sB$) and let $\sV$ be a dense open subset of $\sD$. 
We define the set $T(U,\sV)$ to be the set of points $P \in U(k)$ such that the Zariski closure of $P$ in $\sB$ meets $\sV$ transversely in at least one point.
\end{definition}

Note that, for dimension reasons, the Zariski closure of a point $P \in T(U,\sV)$ meets $\sV$ in finitely many closed points of $\sB$.  A closed point of $\sB$ has finite residue field, that is, it lies in the fibre of $\sB$ above one of the primes of $\fo$.  The definition simply requires that at least one of these intersection points be transverse.

We can now state the proposition.

\begin{proposition}\label{prop:main}
Let $A \subset \H^0(K, \H^2(X_{\bar{\eta}}, \mmu_n))$ be a finite subgroup, and let $d \in B$ be a point of codimension $1$ not contained in $U$.  Let $\sX \to \sB$ be a model of $\pi$ over $\fo$.  Denote by $\sD$ the Zariski closure of $d$ in $\sB$.  Then there exists a dense open subset $\sV \subset \sD$ such that, for all $P \in T(U,\sV)$ such that $X_P$ is everywhere locally soluble,
if we let $\fp$ be a prime of $\fo$ above which the Zariski closure of $P$ meets $\sV$ transversely, then
 the subgroup $A$ is $\{d\}$-prolific at $\{\fp\}$ on the fibre $X_P$.
\end{proposition}

Before proving Proposition~\ref{prop:main}, we deduce some corollaries.  The first is an immediate generalisation to multiple $d$.

\begin{corollary}\label{cor:manyp}
Let $A$ be a finite subgroup of $\H^0(K, \H^2(X_{\bar{\eta}}, \mmu_n))$, and $d_1, \dotsc, d_r$ be finitely many distinct points of codimension $1$ in $B$, not contained in $U$.  For each $i=1, \dotsc, r$, let $\sV_i \subset \sD_i$ be the corresponding dense open set given by Proposition~\ref{prop:main}, and shrink the $\sV_i$ if necessary so that they are disjoint.  
Let $P \in U(k)$ lie in the intersection $T(U,\sV_1) \cap \dotsb \cap T(U,\sV_r)$, suppose that $X_P$ is everywhere locally soluble, and let $\fp_i$ be a prime of $\fo$ above which the Zariski closure of $P$ meets $\sV_i$ transversely.
Then $A$ is $\{d_1, \dotsc, d_r\}$-prolific at $\{\fp_1, \dotsc, \fp_r\}$ on the fibre $X_P$.
\end{corollary}

\begin{proof}
The Zariski closure of $P$ is an $\fo$-point of $\sB$ and so meets the fibre at any prime $\fp$ in only one point (the ``reduction of $P$ modulo $\fp$'').
Because the $\sV_i$ are disjoint, this means that the $\fp_i$ in the statement of the corollary are distinct.
Proposition~\ref{prop:main} shows that the image of $X_P(\fp_i) \to \dual{A}$ contains the image of $\dual{\H^0(\ff{d_i},\H^1(X_{\bar{d_i}}^\sm, \Zn))}$, so the image of the sum of these maps contains
the sum of those images.
\end{proof}

In the second corollary, we replace $\H^0(K,\H^2(X_{\bar{\eta}},\mmu_n))$ by $\H^1(K,\Pic X_{\bar{\eta}})$.
Recall from Section~\ref{sec:res} that composition with the homomorphism~\eqref{eq:h1h2} gives well-defined residue maps
\begin{equation}\label{eq:rhopic}
\rho_d \colon \H^1(K, \Pic X_{\bar{\eta}})[n] \to \H^0(\ff{d}, \H^1(X_{\bar{d}}^\sm, \Zn)).
\end{equation}

Similarly, the specialisation map $\spec_P$ also lifts, by composition with the homomorphism~\eqref{eq:h1h2}, to a homomorphism
\begin{equation}\label{eq:specpic}
\spec_P \colon \H^1(K, \Pic X_{\bar{\eta}})[n] \to \Br X_P[n] / \Br k[n].
\end{equation}
To see this we must show that the image of $\H^0(K, \Pic X_{\bar{\eta}})$ in $\H^0(K, \H^2(X_{\bar{\eta}},\mmu_n))$ lies in the kernel of $\spec_P$.  If $\alpha$ lies in this image, then the specialisation of $\alpha$ in $\H^2(X_P,\mmu_n)/\H^2(k,\mmu_n)$ will lie in the image of $\Pic X_P$, and so will give $0$ when mapped into $\Br X_P[n]/\Br k[n]$, as claimed.  

\begin{corollary}\label{cor:prolific}
Let $A$ be a finite subgroup of $\H^1(K, \Pic X_{\bar{\eta}})[n]$.
Suppose that the homomorphism
\[
A \xrightarrow{\oplus_d \rho_d} \bigoplus_{d \in B^{(1)}} \H^0(\ff{d},\H^1(X_{\bar{d}}^\sm, \Zn))
\]
is injective.  Let $d_1, \dotsc, d_r$ be the finitely many distinct points of codimension $1$ in $B$ satisfying $\rho_{d_i}(A) \neq 0$.  For each $i=1, \dotsc, r$, let $\sV_i \subset \sD_i$ be the corresponding dense open set given by Proposition~\ref{prop:main}, and shrink the $\sV_i$ if necessary so that they are disjoint.
Let $P \in U(k)$ lie in $T(U,\sV_1) \cap \dotsb \cap T(U,\sV_r)$, suppose that $X_P$ is everywhere locally soluble, and let $\fp_i$ be a prime of $\fo$ above which the Zariski closure of $P$ meets $\sV_i$ transversely.
Then the group $\spec_P(A) \subset (\Br X_P[n] / \Br k[n])$ is prolific at $S = \{ \fp_i : i = 1, \dotsc, r \}$, and so gives no Brauer--Manin obstruction to the existence of rational points on $X_P$.
\end{corollary}

\begin{proof}
Firstly, none of the $d_i$ lie in $U$, for the following reason.  By assumption the geometric generic fibre $X_{\bar{\eta}}$ has torsion-free Picard group, and so we have $\H^1(X_{\bar{\eta}}, \Zn) = \H^1(X_{\bar{\eta}}, \mmu_n) = 0$; by proper-smooth base change, the group $\H^1(X_{\bar{s}}^\sm, \Zn) = \H^1(X_{\bar{s}}, \Zn)$ is zero for all geometric points $\bar{s} \in U$.

Now applying Corollary~\ref{cor:manyp}, and remembering that the homomorphisms~\eqref{eq:rhopic} and~\eqref{eq:specpic} both factor through $\H^0(K, \H^2(X_{\bar{\eta}}, \mmu_n))$, we find that the image of
\begin{equation}\label{eq:im1}
\prod_{v \in S} X_P(k_v) \to \dual{(\spec_P(A))} \to \dual{A},
\end{equation}
where the first map is defined by the pairing~\eqref{eq:evalS} and the second map is induced by $\spec_P$,
contains the image of the homomorphism
\begin{equation}\label{eq:im2}
\dual{\bigoplus_{d \in B^{(1)}} \H^0(\ff{d},\H^1(X_{\bar{d}}^\sm, \Zn))} \to \dual{A}
\end{equation}
induced by $\oplus_d \rho_d$.
By assumption, \eqref{eq:im2} is surjective, and so~\eqref{eq:im1} is as well.  Since $\spec_P(A)$ is a quotient of $A$, it follows that $\prod_{v \in S} X_P(k_v) \to \dual{(\spec_P(A))}$ is also surjective, that is, $\spec_P(A)$ is prolific at $S$.
\end{proof}

To begin the proof of Proposition~\ref{prop:main}, we establish some lemmas.
The first introduces a method of associating a torsor to a finite subgroup of $\H^1(T,\QZ)$, for a scheme $T$.  In~\cite[Section~5.3]{Bright:bad}, this was approached by choosing a basis for a subgroup and then taking the fibre product of the corresponding torsors.  The following lemma gives a coordinate-free way of achieving the same result.  It can be viewed as a generalisation of the fact that $\H^1(T,\Zn)$ is canonically the $n$-torsion subgroup of $\H^1(T,\QZ)$.  This lemma is key to applying the ideas of~\cite{BBL:wa} to several Brauer classes at once.

\begin{lemma}\label{lem:torsors}
Let $G$ be a finitely generated Abelian group.
\begin{enumerate}
\item For any scheme $T$, there is a natural isomorphism
\[
\H^1(T,\dual{G}) \to \Hom(G, \H^1(T,\QZ)).
\]
If $T$ is connected then, under this isomorphism, the classes in $\H^1(T,\dual{G})$ represented by connected torsors correspond to the injective homomorphisms $G \to \H^1(T,\QZ)$.
\item For every morphism of schemes $f \colon S \to T$, there is a natural isomorphism
\[
\R^1 f_* \dual{G} \to \sHom_{\catsh(T)}(G, \R^1 f_* \QZ).
\]
\end{enumerate}
\end{lemma}
\begin{proof}
Firstly, suppose that $T$ is a connected scheme.  We have natural isomorphisms
\begin{multline*}
\H^1(T,\dual{G}) = \Hom(\pi_1(T),\Hom(G,\QZ)) \\ = \Hom(G, \Hom(\pi_1(T),\QZ)) = \Hom(G,\H^1(T,\QZ)).
\end{multline*}
It is a standard fact that connected torsors correspond to surjective homomorphisms $\pi_1(T) \to \dual{G}$; by duality these correspond to injective homomorphisms $G \to \dual{\pi_1(T)}$, proving the statement.
If $T$ is not connected, then the isomorphism is obtained by working separately on each connected component.

Now let $f \colon S \to T$ be a morphism.  To avoid confusion, let $G^p$ and $G^s$ denote the constant presheaf and sheaf, respectively, on $T$ defined by $G$.  Let $H$ denote the presheaf on $T_\et$ defined by $H(T') = \H^1(S \times_T T', \QZ)$.
The sheaf $\R^1 f_* \dual{G}$ is the sheafification of the presheaf $F$ on $T_\et$ defined by
\begin{align*}
F(T') &= \H^1(S \times_T T', \dual{G}) \\
&= \Hom(G, \H^1(S \times_T T', \QZ)) & \text{by (1)} \\
&= \Hom(G, H(T')) \\
&= \Hom_{\catpsh(T')}(G^p|_{T'}, H|_{T'})
\end{align*}
the last because the constant presheaf functor is left adjoint to the global sections functor.  By definition of the internal Hom in $\catpsh(T)$, this says $F$ is equal to $\sHom_{\catpsh(T)}(G^p, H)$.  The natural map from $H$ to its sheafification $\R^1 f_* \QZ$ gives a morphism of presheaves 
\[
\psi \colon F \to \sHom_{\catpsh(T)}(G^p, \R^1 f_* \QZ) = \sHom_{\catsh(T)}(G^s, \R^1 f_* \QZ).
\]
The target is a sheaf, so by the universal property of sheafification we obtain a natural morphism of sheaves $\phi \colon \R^1 f_* \dual{G} \to \sHom_{\catsh(T)}(G^s, \R^1 f_* \QZ)$.  To show that $\phi$ is an isomorphism, it is enough to show that $\psi$ induces isomorphisms on stalks.  As $G$ is finitely generated, taking Homs out of $G$ commutes with taking stalks (see \cite[II.3.20]{Milne:EC}), giving the result.
\end{proof}

The following lemma is a straightforward generalisation of Lemma~5.21 of~\cite{BBL:wa}, replacing $\Zn$ by any finite Abelian group.  For a prime ideal $\fp \subset \fo$, the residue field $\fo/\fp$ is denoted by $\FF_\p$.
\begin{lemma}\label{lem:ffpoints}
Let $k$ be a number field, $\fo$ the ring of integers of $k$ and let $\pi \colon Z \to S$ be a dominant morphism of normal, integral, separated $\fo$-schemes of finite type.  Denote by $\eta$ the generic point of $S$ and $\bar{\eta}$ a geometric point lying over $\eta$.  Suppose that $Z_{\bar{\eta}}$ is integral.  Let $G$ be a finite Abelian group and let $\gamma$ be a class in $\H^0(S, \R^1 \pi_* G)$.  For a geometric point $\bar{s} \in S$, let $Y_\gamma(\bar{s})$ denote the torsor over $Z_{\bar{s}}$ defined by the restriction of $\gamma$ to $\H^1(Z_{\bar{s}},G)$.  Suppose that $Y_\gamma(\bar{\eta})$ is connected.  Then there are dense open subsets $S' \subset S$ and $U \subset \Spec \fo$ such that, for every $\fp \in U$ and any $\bar{s} \in S'(\bar{\FF}_\fp)$, any $\FF_\fp$-variety geometrically isomorphic to $Y_\gamma(\bar{s})$ has an $\FF_\fp$-rational point.
\end{lemma}
\begin{proof}
As in the proof of~\cite[Lemma~5.21]{BBL:wa}, we may replace $S$ by an \'etale $T \to S$ and thereby assume that $\gamma$ lifts to $\gamma' \in \H^1(Z,G)$.  Let $f \colon Y \to Z$ be a torsor representing the class $\gamma'$.  
The scheme $Z$ is normal and integral, and so we are in the situation of~\cite[Expos\'e~I, Proposition~10.1]{SGA1}: the connected components of $Y$ are in bijection with the connected components of the generic fibre of $f$.
If the generic fibre of $f$ were not connected, then $Y_\eta$ would also not be connected; but
the fibre $Y_{\bar{\eta}}$ is our original $Y_\gamma(\bar{\eta})$, assumed to be connected.  It follows that $Y$ is also connected; since $Z$ is normal, so is $Y$, and so $Y$ is integral.  
Also, since $Z_{\bar{\eta}}$ is integral, it follows that $Y_{\bar{\eta}}$ is integral, and so $Y_\eta$ is split.  Now applying~\cite[Lemma~3.6]{BBL:wa} to $Y \to S$ gives open sets $S' \subset S$ and $U \subset \Spec \fo$ with the desired property.
\end{proof}

The way that Lemma~\ref{lem:ffpoints} is used in our proof is to show the surjectivity of certain evaluation maps; we now state that application as a corollary.

Suppose that we are in the situation of Lemma~\ref{lem:ffpoints}.
Let $\fp$ be a prime of $\fo$ and let $s$ be any point of $S'(\FF_\fp)$.  Suppose that $Z_s$ is geometrically integral.  (Since $Z_\eta$ is geometrically integral, this happens for $s$ in a dense open subset of $S'$.)
The Hochschild--Serre spectral sequence for the sheaf $G$ on $Z_s$ gives an exact sequence as follows:
\[
0 \to \H^1(\FF_\fp, G) \to \H^1(Z_s,G) \to \H^0(\FF_\fp, \H^1(Z_{\bar{s}},G)) \to \H^2(\FF_\fp, G).
\]
Because $\FF_\fp$ has cohomological dimension $0$, the group $\H^2(\FF_\fp, G)$ is trivial. 
We can therefore obtain a torsor $Y_{\gamma,s} \to Z_s$ by specialising $\gamma$ to obtain an element of $\H^0(\FF_\fp, \H^1(Z_{\bar{s}}, G))$ and then lifting to $\H^1(Z_s,G)$.  Different choices of lifts will give choices of $Y_{\gamma,s}$ differing by twisting by a constant torsor.  Over the algebraic closure of $\FF_\fp$, the variety $Y_{\gamma,s}$ is isomorphic to the variety $Y_\gamma(\bar{s})$ of the Lemma.

\begin{corollary}\label{cor:ffsurj}
Under the conditions of Lemma~\ref{lem:ffpoints}, let $\fp$ lie in $U$ and let $s$ be any point of $S'(\FF_\fp)$ such that $Z_s$ is geometrically integral.  Let $Y_{\gamma,s} \to Z_s$ be a torsor as above.  Then the evaluation map $Z_s(\FF_\fp) \to \H^1(\FF_\fp,G)$ associated to the torsor $Y_{\gamma,s} \to Z_s$ is surjective.
\end{corollary}
\begin{proof}
This is a standard twisting argument.  A point $z \in Z_s(\FF_\fp)$ evaluates to $0\in \H^1(\FF_\fp,G)$ if and only if the fibre $(Y_{\gamma,s})_z$ has an $\FF_\fp$-rational point, that is, $z$ lies in the image of $Y_{\gamma,s}(\FF_\fp) \to Z_s(\FF_\fp)$.  The conclusion of the Lemma ensures that $Y_{\gamma,s}(\FF_\fp)$ is non-empty, and so $0$ lies in the image of the evaluation map.  Twisting $Y_{\gamma,s}$ by any class in $\H^1(\FF_\fp,G)$ and applying the same argument shows that the evaluation map is indeed surjective.
\end{proof}

Having established these lemmas, we are now ready to prove the proposition.  The following proof closely follows that of~\cite[Proposition~4.2]{BBL:wa} in concept, but we must make substantial changes to be able to apply it to finite subgroups of $\H^0(K, \H^2(X_{\bar{\eta}},\mmu_n))$ instead of to single elements.

\begin{proof}[of Proposition~\ref{prop:main}]
To prove the proposition, we may replace $\sB$ by any open subset containing both $U$ and $d$.
By doing so, we can ensure that the following conditions hold:
\begin{itemize}
\item  $n$ is invertible on $\sB$;
\item $\sB$ is regular and flat over $\fo$;
\item $\sD$ is regular;
\item the fibre of $\sX$ above every geometric point of $\sD$ is integral;
\item $A$ is contained in the image of $\H^0(\sB \setminus \sD, \R^2 \pi_* \mmu_n) \to \H^0(K, \H^2(X_{\bar{\eta}},\mmu_n))$.
\end{itemize}
The proof that these conditions may be satisfied after shrinking $\sB$ is routine and details may be found in the proof of~\cite[Proposition~4.2]{BBL:wa}.  From now on, we place ourselves in this situation.
Furthermore, replacing $\sX$ by the open subscheme where $\pi$ is smooth ensures that $\rho_d(A)$ is contained in the image of the natural map $\H^0(\sD, \R^1(\pi_\sD)_* \Zn) \to \H^0(\ff{d}, \H^1(X_{\bar{d}}^\sm, \Zn))$.

Let $P \in U(k)$ be a point such that $X_P$ is everywhere locally soluble.  Let $\sP$ denote the Zariski closure of $P$ in $\sB$ and suppose that $\sP$ meets $\sD$ transversely at a closed point $s$ of $\sB$.  Let $\fp$ be the prime of $\fo$ over which $s$ lies.
We will study the map $X_P(k_\fp) \to \dual{A}$ given by the pairing~\eqref{eq:pairqz}; our goal is to show that the image of this map contains the image of the homomorphism $\dual{\H^0(\ff{d},\H^1(X_{\bar{d}}^\sm, \Zn))} \to \dual{A}$.  We will attempt to prove this, and will show that we can succeed at the expense of assuming that $s$ lies in a certain dense open subset $\sV \subset \sD$.

Functoriality of the residue maps, proved in Proposition~5.6 of~\cite{BBL:wa}, gives us a commutative diagram as follows.

\begin{equation}\label{eq:bigcd}
\begin{CD}
\H^0(K, \H^2(X_{\bar{\eta}}, \mmu_n)) @>{\rho_d}>> \H^0(\ff{d}, \H^1(\sX_{\bar{d}}, \Zn)) \\
@A{\eta^*}AA @A{d^*}AA \\
\H^0(\sB \setminus \sD, \R^2 \pi_* \mmu_n) @>{\rho_\sD}>> \H^0(\sD, \R^1 (\pi_\sD)_* \Zn) \\
@V{P^*}VV @V{s^*}VV \\
\H^0(k, \H^2(X_{\bar{P}}, \mmu_n)) @>{\rho_s}>> \H^0(\ff{s}, \H^1(\sX_{\bar{s}},\Zn)) \\
@A{\sim}AA @A{\sim}AA \\
\H^2(X_P, \mmu_n)/\H^2(k,\mmu_n) @>>> \H^1(\sX_s,\Zn)/\H^1(\ff{s},\Zn) \\
@VVV @| \\
\Br X_P[n]/\Br k[n] @>{\partial}>> \H^1(\sX_s,\Zn)/\H^1(\ff{s},\Zn)
\end{CD}
\end{equation}
Here $\partial$ is the residue map associated to the prime divisor $\sX_s$ on the scheme $\sX_\sP$.
If a class $\alpha$ lies in the top left-hand group, then $\spec_P(\alpha)$ is its image in the bottom left-hand group.

Now consider evaluating a class in $\Br X_P$ at points of $X_P(k_\fp)$.  
Write $\FF$ for the residue field $\ff{s} = \fo/\fp$. 
Since $\pi$ is no longer assumed to be proper, not all such points extend to $\fo$-points of $\sX$; write $X_P(k_\fp)^\circ$ for the set of those that do.  Because $\pi$ is smooth, Hensel's Lemma shows that the reduction map $X_P(k_\fp)^\circ \to \sX_s(\FF)$ is surjective.  We assume for now that $X_P(k_\fp)^\circ$ is non-empty; later we will choose $P$ to ensure that this is the case.  Fix a base point $x_0 \in X_P(k_\fp)^\circ$.
Proposition~5.1 of~\cite{Bright:bad} and its corollaries give a commutative diagram
\[
\begin{CD}
\Br X_P[n] @. \;\times\; @. X_P(k_\fp)^\circ @>>> \Br k_\p[n] \\
@V{\partial}VV @. @VVV @V{\sim}VV \\
\H^1(\sX_s,\Zn) @. \;\times\; @. \sX_s(\FF) @>>> \H^1(\FF,\Zn).
\end{CD}
\]
It follows that the induced map $X_P(k_\p)^\circ \to \Hom(\Br X_P[n], \Br k_\p[n])$ factors as
\begin{multline*}
X_P(k_\p)^\circ \to \sX_s(\FF) \to \Hom(\H^1(\sX_s,\Zn), \H^1(\FF,\Zn)) \\
\to \Hom(\Br X_P[n], \Br k_\p[n]),
\end{multline*}
where the first map is reduction modulo $\p$, the second sends a point to the evaluation map at that point, and the third is given by pre-composing with $\partial$ and post-composing with the inverse of the natural isomorphism $\Br k_\p[n] \to \H^1(\FF,\Zn)$.  From this, and from commutativity of the diagram~\eqref{eq:bigcd}, we deduce that the map $X(k_\p)^\circ \to \Hom(A,\Br k_\p[n])$ defined by the pairing~\eqref{eq:pairbr} similarly factors as
\begin{equation}\label{eq:comp}
X_P(k_\p)^\circ \to \sX_s(\FF) \xrightarrow{e_P} \Hom(\rho_d(A), \H^1(\FF,\Zn)) \to \Hom(A,\Br k_\p[n]).
\end{equation}
Here the map $e_P$ can be described as follows.  Any $\beta$ lying in the image $\rho_d(A)$ extends to $\H^0(\sD, \R^1(\pi_\sD)_* \Zn)$ and so can be specialised at $s$ to obtain a class in $\H^0(\FF,\H^1(\sX_{\bar{s}},\Zn))$.  Lift this class to $\beta_s \in \H^1(\sX_s,\Zn)$.  Then $e_P$ sends a point $y \in \sX_s(\FF)$ to the homomorphism $\beta \mapsto (\beta_s(y) - \beta_s(\tilde{x}_0))$, where $\tilde{x}_0$ is the reduction of $x_0$ modulo $\p$.  This is easily seen to be independent of the lift $\beta_s$.

The right-hand two groups in~\eqref{eq:comp} are canonically identified with the dual groups $\dual{(\rho_d(A))}$ and $\dual{A}$, respectively, by means of the canonical isomorphisms $\H^1(\FF,\Zn) \isom \Br k_\p[n]$ and $\Br k_\p \isom \QZ$.
To prove the proposition, we will exhibit a dense open subset $\sV \subset \sD$ such that, for $P$ as in the statement of the proposition, the image of the composition~\eqref{eq:comp} coincides with the image of the rightmost homomorphism.  The first map is surjective by Hensel's Lemma, so it suffices to show that $e_P$ is surjective.  To accomplish this, we use Lemma~\ref{lem:torsors} to pass from a subgroup of $\H^1(\sX_{\bar{d}},\Zn)$ to a single torsor under a more complicated group, to which Lemma~\ref{lem:ffpoints} will apply.


Write $C=\rho(A)$.  Applying Lemma~\ref{lem:torsors} gives
\[
\Hom(C, \H^1(\sX_{\bar{d}}, \QZ))) \isom \H^1(\sX_{\bar{d}}, \dual{C}).
\] 
Now take Galois invariants (which commutes with taking $\Hom$ from $C$) to obtain
\[
\Hom(C, \H^0(\ff{d}, \H^1(\sX_{\bar{d}}, \QZ))) \isom \H^0(\ff{d}, \H^1(\sX_{\bar{d}}, \dual{C})).
\]
The inclusion map of $C$ in $\H^1(\sX_{\bar{d}},\Zn)$, composed with the natural injection $\H^1(\sX_{\bar{d}},\Zn) \to \H^1(\sX_{\bar{d}}, \QZ)$, gives an element of the left hand side, and so corresponds to an element $\gamma$ of the right hand side.  
This gives a connected torsor $Y_{\bar{d}} \to \sX_{\bar{d}}$ under $\dual{C}$.

By construction, $C$ lies in the image of $\H^0(\sD, \R^1(\pi_\sD)_*\Zn)$; the functoriality of the isomorphisms of Lemma~\ref{lem:torsors} then shows that $\gamma$ extends to an element of $\H^0(\sD, \R^1 (\pi_\sD)_* \dual{C})$.  The conditions of Lemma~\ref{lem:ffpoints} are satisfied, and we obtain dense open subsets $S' \subset \sD$ and $U \subset \Spec \fo$ as in that lemma.  Let $\sV \subset \sD$ be the intersection of $S'$ with the inverse image of $U$.  We will now assume that $s$ lies in $\sV$, and prove that $e_P$ is indeed surjective.

Specialising $\gamma$ at $s$ and choosing a lift to $\H^1(\sX_s,\dual{C})$ gives a torsor $Y_s \to \sX_s$.  
Corollary~\ref{cor:ffsurj} shows that the evaluation map associated to the torsor $Y_s \to \sX_s$ is surjective.
(In particular, $\sX_s(\FF)$ and therefore $\sX_P(k_\fp)^\circ$ are non-empty, as assumed earlier.)
It follows that the map
\begin{equation}\label{eq:Ymap}
\sX_s(\FF) \to \H^1(\FF,\dual{C}), \quad y \mapsto Y_s(y) - Y_s(\tilde{x}_0),
\end{equation}
which is independent of the choice of lift used to define $Y_s$, is also surjective.
Composing with the isomorphism $\H^1(\FF,\dual{C}) \to \Hom(C,\H^1(\FF,\QZ))$ of Lemma~\ref{lem:torsors} gives a map $\sX_s(\FF) \to \Hom(C,\H^1(\FF,\Zn))$, where as usual we identify the group $\H^1(\FF,\Zn)$ with the $n$-torsion subgroup of $\H^1(\FF,\QZ)$.  The functoriality of Lemma~\ref{lem:torsors} shows that this map is none other than $e_P$.  So $e_P$ is indeed surjective, proving the proposition.
\end{proof}

\section{Proof of Theorem~\ref{thm}}\label{sec:sieve}

To deduce Theorem~\ref{thm} from Corollary~\ref{cor:prolific}, we must show that the hypotheses are satisfied when $A$ is the whole of $\H^1(K, \Pic X_{\bar{\eta}})$, and then show that 100\% of points $P \in U(k)$ satisfy the condition required to establish vanishing of the Brauer--Manin obstruction.  The conclusion of Theorem~\ref{thm} is vacuously true if $\H^1(K, 
\Pic X_{\bar{\eta}})$ is trivial, so assume that this is not the case.

Firstly, the conditions $(\star)$ demanded throughout Section~\ref{sec:surj} do indeed follow from the assumptions of Theorem~\ref{thm}.  Specifically, it is here that we use conditions~(\ref{codim1}) and~(\ref{codim2}) of Theorem~\ref{thm}.

The requirement that $\Pic X_{\bar{\eta}}$ be torsion-free implies that $\H^1(X_{\bar{\eta}}, \OO_{X_{\bar{\eta}}})$ is trivial; therefore the Picard group of $X_{\bar{\eta}}$ is the same as its N\'eron--Severi group, and so is finitely generated and free.  Therefore $\H^1(K, \Pic X_{\bar{\eta}})$ is finite; fix a positive integer $n$ such that every element of $\H^1(K, \Pic X_{\bar{\eta}})$ has order dividing $n$.  
Let $\sX \to \sB = \PP^m_\fo$ be a model of $\pi$ over $\fo$.

By Corollary~5.16 of~\cite{BBL:wa}, the residue homomorphism
\[
\oplus_d \rho_d \colon \H^1(K, \Pic X_{\bar{\eta}}) \to \bigoplus_{d \in (\PP^m_k)^{(1)}} \H^0(\ff{d}, \H^1(X_{\bar{d}}^\sm,\Zn))
\]
is injective.  It is here that we use conditions~(\ref{h1zero}), (\ref{brzero}) and~(\ref{h2inj}) of Theorem~\ref{thm}. 
So Corollary~\ref{cor:prolific} applies.  Let $d_1, \dotsc, d_r$ be the finitely many points $d \in \PP^m_k$ of codimension $1$ for which the image of $\rho_{d_i}$ is non-zero.  We obtain, for each $i$, a non-empty open subset $\sV_i \subset \sD_i$, where $\sD_i$ is the Zariski closure of $d_i$ in $\sB = \PP^m_\fo$.

Let $T$ denote the set $U(k) \cap \pi(X(\Adele_k))$, and let $T_i$ be the set of points $P \in T$ such that the Zariski closure of $P$ in $\PP^m_\fo$ meets $\sV_i$ transversely in at least one point.  To prove Theorem~\ref{thm}, we must show that $\bigcap_{i=1}^r T_i$ contains 100\% of the points of $U(k)$.  Proposition~4.3 of~\cite{BBL:wa} states that the number of points of $T \setminus T_i$ of height at most $B$ is $O(B^{m+1}/\log B)$. We have
\begin{multline*}
\# \{ P \in T : H(P) \le B \} - \#\{ P \in \bigcap _{i=1}^r T_i : H(P) \le B \} \\
\le \sum_{i=1}^r \#\{ P \in T \setminus T_i : H(P) \le B \}  
= O \left( \frac{B^{m+1}}{\log B} \right).
\end{multline*}

On the other hand, Theorem~1.3 of~\cite{BBL:wa} states that the number of rational points $P \in \pi(X(\Adele_X))$ of height at most $B$ is asymptotic to a constant times $B^{m+1}$.  (It is here that we use condition~(\ref{locsol}) of Theorem~\ref{thm}.)  Since $U$ is an open subset in $\PP^m$, its complement is closed and therefore thin; then~\cite[Section~13.1, Theorem~3]{Serre:LMWT} shows that the number of points in $T$ of height at most $B$ is also asymptotic to a constant times $B^{m+1}$.  (Note that Serre's height is normalised differently to ours, accounting for the factor of $d=[k:\QQ]$ appearing in Serre's formula.)  Putting these together completes the proof of Theorem~\ref{thm}.

\section{Surjectivity of specialisation}\label{sec:spec}

In this section we prove Theorem~\ref{thm2}.  A subset $H \subset \PP^m(k)$ is said to be \emph{Hilbertian} if there is a non-empty Zariski open subset $W \subset \PP^m_k$ and a finite \'etale morphism $Y \to W$, with $Y$ integral, such that $H$ consists of those points $P \in W(k)$ for which the fibre $Y_P$ is connected.  In~\cite{Harari:DMJ-1994}, Harari proved a result on specialisation of $\H^1(K, \Pic X_{\bar{\eta}})$, which we now adapt to our situation.  Recall that $U \subset \PP^m_k$ is the open subset over which the fibres of $\pi$ are smooth.

\begin{lemma}
Under the hypotheses of Theorem~\ref{thm2}, there exists a Hilbertian subset $H$ of $U(k)$ such that, for each $P \in H$ with $X_P$ everywhere locally soluble, the composition of the homomorphism~\eqref{eq:h1h2} with $\spec_P$ gives an isomorphism of $\H^1(K, \Pic X_{\bar{\eta}})$ with $\Br X_P[n] / \Br k[n]$.
\end{lemma}
\begin{proof}
For $P \in U(k)$, let $\tilde{U}_P$ be the spectrum of the local ring of $U$ at $P$, and $\tilde{X}_P$ the base change of $X$ to $\tilde{U}_P$.  Harari~\cite[Th\'eor\`eme~3.5.1 et seq.]{Harari:DMJ-1994} has shown that there is a Hilbertian subset $H$ of $U(k)$ such that, if $P$ lies in $H$, then the composition $\Pic X_{\bar{\eta}} \leftarrow \Pic \tilde{X}_P \to \Pic \bar{X}_P$ defines a specialisation homomorphism inducing isomorphisms $\Pic X_{\bar{\eta}} \to \Pic \bar{X}_P$ and $\H^1(K, \Pic X_{\bar{\eta}}) \to \H^1(k, \Pic \bar{X}_P)$.  One checks that the diagram
\[
\begin{CD}
\H^1(K, \Pic X_{\bar{\eta}})[n] @>>> \H^0(K, \H^2(X_{\bar{\eta}}, \mmu_n)) / \H^0(K, \Pic X_{\bar{\eta}}) \\
@VVV @VVV \\
\H^1(k, \Pic \bar{X}_P)[n]  @>>> \H^0(k, \H^2(\bar{X}_P, \mmu_n)) / \H^0(k, \Pic \bar{X}_P)\\
\end{CD}
\]
commutes, where the right-hand map is as at the beginning of Section~\ref{sec:surj}.

Now suppose that $X_P$ is everywhere locally soluble and consider the following diagram, arising from the Kummer sequence on both $X_P$ and $\bar{X}_P$.
\[
\minCDarrowwidth1.2em
\begin{CD}
0 @>>> \Pic X_P/n @>>> \H^2(X_P,\mmu_n) @>>> \Br X_P[n] @>>> 0 \\
@. @V{f}VV @V{g}VV @V{h}VV \\
0 @>>> \H^0(k, \Pic \bar{X}_P/n) @>>> \H^0(k, \H^2(\bar{X}_P,\mmu_n)) @>>> \H^0(k, \Br \bar{X}_P)[n]
\end{CD}
\]

Because $X_P$ is everywhere locally soluble we have $\H^0(k,\Pic \bar{X}_P) = \Pic X_P$.  Also, $\Pic \bar{X}_P$ is torsion-free, and so looking at the multiplication-by-$n$ map on $\Pic \bar{X}_P$ shows that $f$ is injective with cokernel $\H^1(k, \Pic \bar{X}_P)[n]$.   Furthermore, \cite[Lemma~5.20]{BBL:wa} shows that $g$ is surjective, with kernel $\Br k[n]$.  We have $\Br \bar{X}_P=0$, as can be found, for example, in the proof of the above theorem of Harari.  So the Snake Lemma gives an exact sequence
\[
0 \to \Br k[n] \to \Br X_P[n] \xrightarrow{\delta} \H^1(k,\Pic \bar{X}_P)[n] \to 0
\]
and looking at the definition of $\delta$ shows that the composite isomorphism
\[
\H^1(K, \Pic X_{\bar{\eta}})[n] \to \H^1(k, \Pic \bar{X}_P)[n] \xrightarrow{\delta^{-1}} \Br X_P[n] / \Br k[n]
\]
is precisely the isomorphism claimed.
\end{proof}

\begin{remark}
Presumably the map $\delta$ in the above proof is the same as that induced by the homomorphism $\Br_1 X_P \to \H^1(k, \Pic \bar{X}_P)$ in the Hochschild--Serre spectral sequence, but we have not verified this.
\end{remark}

A Hilbertian subset of $U(k)$ is the complement of a thin set, by~\cite[Section~9.2, Proposition~2]{Serre:LMWT}.  By~\cite[Section~13.1, Theorem~3]{Serre:LMWT}, $H$ contains 100\% of the points of $U(k)$, and Theorem~\ref{thm2} follows.

\bibliographystyle{abbrv}
\bibliography{martin}

\end{document}